\documentclass[a4paper,leqno]{amsart}

\usepackage{ifdraft}
\ifdraft{\usepackage[draft]{showkeys}}{\usepackage[final]{showkeys}}

\usepackage[usenames,dvipsnames]{xcolor} 
\usepackage{todonotes}

\usepackage[T1]{fontenc}
\usepackage[utf8]{inputenc}
\usepackage{lmodern}

\usepackage[final]{microtype}

\usepackage{mathtools}
\mathtoolsset{mathic}

\usepackage{amsmath,amsfonts,amsthm,mathscinet,eucal,amssymb,bm,extarrows,upgreek,tensor,mathrsfs,textcomp,comment,bbm,braket}

\usepackage[shortlabels]{enumitem}

\usepackage{tikz}
\usetikzlibrary{patterns,matrix,through,arrows,decorations.pathreplacing,decorations.markings,shadows,shapes.geometric,positioning,calc,backgrounds,fit}

\tikzstyle directed=[postaction={decorate,decoration={markings,
    mark=at position #1 with {\arrow{>}}}}]
\tikzstyle rdirected=[postaction={decorate,decoration={markings,
    mark=at position #1 with {\arrow{<}}}}]

\tikzset{anchorbase/.style={baseline={([yshift=-0.5ex]current bounding box.center)}},
anchorzero/.style={baseline={([yshift=-0.5ex]0,0)}},
arrowinthemiddle/.style={postaction=decorate,decoration={markings,mark=at position 0.5 with {\arrow{>}}}},
arrowinthemiddlerev/.style={postaction=decorate,decoration={markings,mark=at position 0.5 with {\arrow{<}}}},
cross line/.style={preaction={draw=white,line width=4pt,-}},
int/.style={thick},
zero/.style={thin,dotted},
uno/.style={thin}}

\usepackage{graphicx}
\usepackage{wrapfig}
 \usepackage[font=sf, labelfont={sf,bf}, margin=1cm]{caption}
\usepackage{subcaption}

 \usepackage[strict]{changepage}

\usepackage{graphicx}

\usepackage[final=true]{hyperref,bookmark}

\numberwithin{equation}{section}

\newtheoremstyle{myplain} {6pt plus 6pt minus 2pt}
{6pt plus 6pt minus 2pt}
{\itshape}
{}
{\bfseries}
{.}
{.5em}
{}

\theoremstyle{myplain}
\newtheorem{theorem}{Theorem}[section]
\newtheorem*{theorem*}{Theorem}
\newtheorem{lemma}[theorem]{Lemma}
\newtheorem{prop}[theorem]{Proposition}
\newtheorem{corollary}[theorem]{Corollary}

\newtheoremstyle{mydefinition} {6pt plus 6pt minus 2pt}
{6pt plus 6pt minus 2pt}
{\itshape}
{}
{\bfseries}
{.}
{.5em}
{}

\theoremstyle{mydefinition}
\newtheorem{definition}[theorem]{Definition}

\newtheoremstyle{myexample} {6pt plus 6pt minus 2pt}
{6pt plus 6pt minus 2pt}
{}
{}
{\scshape}
{.}
{.5em}
{}

\theoremstyle{myexample}
\newtheorem{example}[theorem]{Example}

\newtheoremstyle{myremark} {6pt plus 6pt minus 2pt}
{6pt plus 6pt minus 2pt}
{}
{}
{\scshape}
{.}
{.5em}
{}

\theoremstyle{myremark}
\newtheorem{remark}[theorem]{Remark}

\newtheoremstyle{citing}
{3pt}{3pt}
{\itshape}
{0pt}{\bfseries}
{.}
{ }
{\thmnote{#3}}
\theoremstyle{citing}
\newtheorem*{varthm}{}

\usepackage[bbgreekl]{mathbbol}

\DeclareSymbolFontAlphabet{\mathbb}{AMSb}
\DeclareSymbolFontAlphabet{\mathbbol}{bbold}
\DeclareMathAlphabet{\mathpzc}{OT1}{pzc}{m}{it}

\DeclareSymbolFont{usualmathcal}{OMS}{cmsy}{m}{n}
\DeclareSymbolFontAlphabet{\mathucal}{usualmathcal}

\newcommand{\N}{\mathbb{N}}
\newcommand{\Z}{\mathbb{Z}}

\newcommand{\C}{\mathbb{C}}

\newcommand{\gl}{\mathfrak{gl}}

\newcommand{\mapto}{\rightarrow}

\newcommand\mapsfrom{\mathrel{\reflectbox{\ensuremath{\mapsto}}}}

\DeclareMathOperator{\Hom}{Hom}

\DeclareMathOperator{\End}{End}

\renewcommand{\epsilon}{\varepsilon}

\renewcommand{\phi}{\varphi}

\newcommand{\calS}{{\mathcal{S}}}

\newcommand{\sfP}{{\mathsf{P}}}

\newcommand{\lattice}{\sfP}
\newcommand{\zeroweight}{\boldsymbol{0}}

\newcommand{\onel}{\boldsymbol{1}_{\lambda}}
\newcommand{\onell}[1]{\boldsymbol{1}_{#1}}

\newcommand{\qbin}[2]{\genfrac{[}{]}{0pt}{}{#1}{#2}}

\newcommand{\U}{\dot{{\bf U}}}

\newcommand{\sln}{\mathfrak{sl}_n}
\newcommand{\slm}{\mathfrak{sl}_m}
\newcommand{\slnn}[1]{\mathfrak{sl}_{#1}}

\newcommand{\glk}{\mathfrak{gl}_k}
\newcommand{\gll}{\mathfrak{gl}_l}
\newcommand{\glnn}[1]{\mathfrak{gl}_{#1}}

\newcommand{\Uzmm}[1]{\U_q^{\geq 0}(\glnn{#1})}

\newcommand{\Uzmk}{\U_q^{\geq 0, \leq m}(\glk)}

\newcommand{\Uznnmm}[2]{\U_q^{\geq 0, \leq #1}(\glnn{#2})}

\newcommand{\UzmmB}[1]{\U_q(\glnn{#1})_{\beta}}

\newcommand{\Uzk}{\U_q^{\geq 0}(\glk)}
\newcommand{\Uzl}{\U_q^{\geq 0}(\gll)}

\newcommand{\UzklB}{\U_q(\glnn{k+l})_{\beta}}
\newcommand{\Uzerokl}{{\U_q(\glnn{k+l})_{0}}}
\newcommand{\Uzerokk}[1]{{\U_q(\glnn{#1})_{0}}}

\hypersetup{
  colorlinks = true,
  urlcolor = blue,
  pdfauthor = {Hoel Queffelec and Antonio Sartori},
  pdfkeywords = {Howe duality, gl(1|1)},
  pdftitle = {},
  pdfsubject = {},
  pdfpagemode = UseNone,
  bookmarksopen = true,
  bookmarksopenlevel = 3,
  pdfdisplaydoctitle = true }

\title[Doubled Schur algebras]{Homfly-Pt and Alexander polynomials from a doubled Schur algebra}

\author{Hoel Queffelec}
\address{Mathematical Sciences Institute, Australian National University, John Deedman Building, 27 Union Lane, Canberra ACT 0200}
\email{hoel.queffelec@anu.edu.au}

\author{Antonio Sartori}
\address{Mathematisches Institut, Albert-Ludwigs-Universität Freiburg,
Eckerstraße 1, 79104 Freiburg im Breisgau, Germany }
\email{antonio.sartori@math.uni-freiburg.de}

\begin{document}

\begin{abstract}
We define a generalization of the idempotented Schur algebra which gives a unified setting for a  quantum group presentation of the Homfly-Pt polynomial, together with its specializations to the Alexander polynomial and the \(\slm\) Reshetikhin-Turaev invariants.
\end{abstract}
\maketitle

\setcounter{tocdepth}{1}
\tableofcontents

%
\section{Introduction}
\label{sec:Intro}
%

If two knot invariants were to be cited, they would surely be the Alexander and Jones polynomials. The latter gave rise to a whole family of quantum invariants \cite{MR1036112}, built from maps of representations of the quantum group \(U_q(\slm)\), and has been the ground stone of quantum topology. Although originating from a much more topological setting, the Alexander polynomial, whose developments are as well still investigated, is known to admit quantum group based descriptions (see \cite{1149.57024,SarAlexander} and references therein).

A unifying generalization of these two polynomials is provided by the two-variable Homfly-Pt polynomial \cite{Homfly,PT}. Unfortunately, the quantum group descriptions don't seem to extend very well to this new context, mostly because the expected appearance of infinite-dimensional representations (for example of \(U_q(\slnn{\infty})\)) makes it impossible to define some of the basic intertwiners.

An alternative, or, more precisely, dual description of the \(\slm\) polynomials has been recently introduced by Cautis, Kamnitzer and Morrison \cite{CKM} using skew-Howe duality. In a Schur-Weyl-flavored process, the intertwiners for the \(U_q(\slm)\) representations involved in the construction of the knot invariants are described as images of elements of a dual quantum group \(U_q(\glk)\). As a consequence, they were able to give the first complete description by generators and relations of the monoidal subcategory of \(U_q(\slm)\) representations generated by exterior powers of the vector representation. A similar dual description using skew Howe duality is possible also for the Alexander polynomial, see \cite{Grant} and \cite{miophd2}. In both cases, the intertwiners actually live in the Schur algebra associated to the quantum group \(U_q(\glk)\).

In this paper, the first of a series of two \cite{QS2}, we introduce a doubling of this algebra specifically designed to provide a unified representation theoretical setting for the Homfly-Pt polynomial and all its specializations. This method is likely to give a powerful approach to the investigations on spectral sequences categorifying this phenomenon. The dual approach we adopt here should more generally shed light on representation-theory-based categorifications of the Homfly-Pt polynomial and of the Alexander polynomial. The latter could be particularly interesting for finding a link with Heegaard-Floer homology.
This first short paper is devoted to defining the doubled Schur algebra \(\UzklB\) and giving a minimal working example of the associated link invariants, entirely in the Howe dual context.

We use a slightly modified version of the usual Homfly-Pt polynomial, which is the two-variable framed  link invariant defined by the  skein formula
\[
\begin{tikzpicture}[scale=.75,anchorbase]
\draw [thick,->] (0,0) -- (1,1);
\draw [thick] (1,0) -- (.6,.4);
\draw [thick,->] (.4,.6) -- (0,1);
\end{tikzpicture}
\;-\;
\begin{tikzpicture}[scale=.75,anchorbase]
\draw [thick,->] (1,0) -- (0,1);
\draw [thick] (0,0) -- (.4,.4);
\draw [thick,->] (.6,.6) -- (1,1);
\end{tikzpicture}
\; = \; (q^{-1}-q)\;\; 
\begin{tikzpicture}[scale=.75,anchorbase]
\draw [thick,->] (0.8,0) .. controls ++(-0.2,0.3) and ++(-0.2,-0.3) .. ++(0,1);
\draw [thick,->] (0,0) .. controls ++(0.2,0.3) and ++(0.2,-0.3) .. ++(0,1);
\end{tikzpicture}
\]
and the following value for the unknot:
\begin{equation*}
\begin{tikzpicture}[scale=.75,anchorbase]
\draw [thick,directed=0.5] (0,0) circle (0.6 cm);
\end{tikzpicture}
\; = \frac{q^{\beta}-q^{-\beta}}{q-q^{-1}},
\end{equation*}
where \(q^{\beta}\) is a formal parameter.
A reduced version version of this polynomial can be obtained by dividing by the value of the unknot. Both the \(\slm\) invariants and the Alexander polynomials can be obtained as specializations of the Homfly-Pt polynomial.

We develop a process to assign to the closure \(\hat{\boldsymbol{b}}\) of a braid \(\boldsymbol{b}\) an element \(P(\boldsymbol{b})\) of the doubled Schur algebra \(\UzklB\). Depending on the value of \(\beta\), we recover either the Homfly-Pt polynomial or the Reshetikhin-Turaev invariant:

\begin{varthm}[Theorem~\ref{thm:HomflyPt}]
If \(\beta\) is generic, then \(P(\boldsymbol{b})\) is the Homfly-Pt polynomial of the link \(\hat{\boldsymbol{b}}\). If \(\beta = m \in \Z\) with \(m \geq 2\), then \(P(\boldsymbol{b})\) is the \(\mathfrak{sl}_m\) link invariant of \(\hat{\boldsymbol{b}}\).
\end{varthm}

The Alexander polynomial can also be obtained by considering a suitable version of the doubled Schur algebra:

\begin{varthm}[Theorem~\ref{prop:Alexander}]
Let \(\boldsymbol{b}\) be a braid whose closure is a link \(K\). Let \(\tilde{\boldsymbol{b}}\) denote the closure of all but one strands of \(\boldsymbol{b}\). Then the image of \(\tilde{\boldsymbol{b}}\) in \(\End_{\Uzerokk{k+k-1}}(\onell{(0,\dots, 0,1,0,\dots, \beta-0,\dots)})\), where the \(1\) appears in the same position as the open strand, is the Alexander polynomial of \(K\).
\end{varthm}

Finally, the same process in the  cases \(\beta\) generic or \(\beta=m \geq 2\) yields the reduced version of the corresponding polynomials, giving us a unified quantum group setting for all these polynomials and their interactions.

\begin{varthm}[Proposition~\ref{prop:reducedSpecialization}]
Let \(K\) be a link, and let \(\tilde{\boldsymbol{b}}\) be as in Theorem~\ref{prop:Alexander}. Then
\begin{enumerate}[(i)]
\item the image of \(\tilde{\boldsymbol{b}}\)
  in
  \(\End_{\UzmmB{(k)+(k-1)}}(\onell{(0,\dots, 0,1,0,\dots,
    \beta-0,\dots,\beta-0)})\),
   gives the
  reduced Homfly-Pt polynomial of \(K\), and
\item  the image of \(\tilde{\boldsymbol{b}}\) 
 in
  \(\End_{\U_q(\gl_{(k)+(k-1)})_{\beta=m}}(\onell{(0,\dots,
    0,1,0,\dots, \beta-0,\dots,\beta-0)})\)
  gives the reduced \(\slm\) invariant of \(K\).
\end{enumerate}
\end{varthm}

We save for the sequel \cite{QS2} the investigation of the representation-theoretical version of this story and its consequences, as well as the full generalization of the invariants to tangles, which best comes after exploring the formal properties of the doubled Schur algebra. In particular, our construction allows us to use both exterior powers of vector representations and their duals in skew-Howe duality, and gives unifying results about the \(\glnn{m|n}\) super-invariants.

\vspace{1em}
\noindent {\bf Acknowledgements:} We would like to thank Christian Blanchet, Mikhail Khovanov, Aaron Lauda, David Rose, Peng Shan and Catharina Stroppel for helpful discussions. H.Q. was funded by ARC DP 140103821.

%
\section{Doubled Schur algebra}
\label{sec:DbleQGroup}
%

In the following, by \(\beta\) we will denote either an integer number or a formal variable (in the last case, we will say that \(\beta\) is \emph{generic}).
Let \(\C(q)\) denote the
field of rational functions in the variable \(q\). We will work over the field \(\C(q,q^{\beta})\).
If \(\beta \in \Z\) then \(q^\beta \in \C(q)\) and we have \(\C(q,q^\beta) = \C(q)\).
 If \(\beta\) is generic, we let \(\C(q,q^\beta)\) be a transcendental extension of \(\C(q)\) by the element \(q^\beta\).
We will allow natural algebraic manipulations, for example we will write \(q^{k+\beta}\) for \(q^k q^\beta\).

Let \(\lattice=\lattice^\beta_{k,l}\) be the set of sequences \(\lambda=(\lambda_1,\dotsc,\lambda_{k+l})\) with \(\lambda_i \in \N\) for \(i\leq k\) and \(\lambda_i \in \beta-\N\) for \(i\geq k+1\). We let also \(\alpha_i = (0,\dotsc,0,1,-1,0,\dotsc,0)\), the entry \(1\) being at position \(i\).

\begin{definition} \label{defn_UzklB}
We define \(\UzklB\) to be the \(\C(q,q^\beta)\)-linear category with:
\begin{itemize}
\item {\bf objects:} \(\onell{\lambda}\) for \(\lambda \in \lattice\) and a zero object;
\item {\bf morphisms:} generated by identity endomorphisms \(\onel\) in \(\Hom(\onell{\lambda},\onell{\lambda})\), and morphisms \(E_i\onel=\onell{\lambda+\alpha_i}E_i \in \Hom(\onel,\onell{\lambda+\alpha_i})\), \(F_i\onel=\onell{\lambda-\alpha_i}F_i  \in \Hom(\onel, \onell{\lambda-\alpha_i})\).
We will often abbreviate by omitting some of the symbols \(\onel\). The morphisms are subject to the following relations:
\begin{enumerate}[itemsep=5pt,topsep=5pt]
\item \label{item:1} \(\onel\onell{\lambda'}=\delta_{\lambda,\lambda'}\onel\);
\item \label{item:4} \([E_i,F_j]\onel=\delta_{i,j}[\lambda_{i}-\lambda_{i+1}]\onel\);
\item \label{item:6} \(\begin{aligned}[t]
E_i^2E_j\onel-(q+q^{-1})E_iE_jE_i\onel+E_jE_i^2\onel& =0  &\text{if } j=i\pm 1,  \\
F_i^2F_j\onel-(q+q^{-1})F_iF_jF_i\onel+F_jF_i^2\onel&=0 &\text{if } j=i\pm 1,  \\
E_iE_j\onel=E_jE_i\onel, \quad F_iF_j\onel = F_jF_i\onel & &\text{if }|i-j|>1. 
\end{aligned}\)
\end{enumerate}
\end{itemize}
\end{definition}

We indifferently use: \(\onell{\lambda+\alpha_i}E_i\onel=E_i\onel=\onell{\lambda+\alpha_i}E_i\), since knowing the source or the target of the \(1\)-morphism is enough to determine the other one. Note that in some of the relations, some weights may not be admissible, while other ones are. The non-admissible ones are then sent to the zero object, and this convention will be used throughout the paper.

One may also introduce the divided powers \(E_i^{(a)}\onel=\frac{1}{[a]!}E_i^{a}\onel\), and similarly for the \(F_i\)'s. As in \cite[\textsection{}7.1.1]{Lus4}, the following higher quantum Serre relation follows from \eqref{item:6} in Definition~\ref{defn_UzklB}: for \(m\), \(n\) with 
\(m \geq n+1\)
we have:
\begin{equation}\label{eq:1}
F_1^{(m)}F_2^{(n)}=\sum_{\substack{r+s=m,\\m- n\leq s\leq m}} \gamma_s F_1^{(r)} F_2^{(n)} F_1^{(s)}
\end{equation}
where
\begin{equation}
 \gamma_s=\sum_{t=0}^{m- n -1}(-1)^{s+1+t}q^{-s( n-m+1+t)} \qbin{s}{t}.\label{eq:2}
\end{equation}
Similarly, we can deduce from \eqref{item:4}: 
\begin{align} \label{eq:higherEF}
E_i^{(a)}F_i^{(b)}\onel=&\sum_{t=0}^{\min\{a,b\}} \qbin{a-b+\lambda_i-\lambda_{i+1}}{ t} F_i^{(b-t)}E_i^{(a-t)}\onel,  \\ F_i^{(b)}E_i^{(a)}\onel=&\sum_{t=0}^{\min\{a,b\}} \qbin{-a+b-\lambda_i+\lambda_{i+1}}{ t} E_i^{(a-t)}F_i^{(b-t)}\onel. \label{eq:18}
\end{align}

One can define an integral version generated by
  the divided powers by replacing \eqref{item:4} with
  \eqref{eq:higherEF}, \eqref{eq:18} and \eqref{item:6} with
  \eqref{eq:1}. Actually, we will use the divided powers for defining the link invariants.

If \(\beta=m \in \N\) then \(\UzklB\) is isomorphic to the dot version  of the quantum enveloping algebra of \(\mathfrak{gl}_m\). In particular, it is non trivial, and one can produce a basis of it. 
This can be used to show that \(\UzklB\) is also non trivial and that \(\UzklB\) contains as a subalgebra \(\U_q(\gl_{k'+l'})_\beta\) for \(k'\leq k\) and \(l \leq l'\).

\begin{remark}
The \(l=0\) specialization of the above definition recovers the usual category \(\Uzk=\bigoplus_{N\geq 0}\calS(k,N)\), where \(\calS(k,N)\) is the quantum Schur algebra of degree \(N\) associated to \(\glk\) (see \cite{MR1920169}).
\end{remark}

\begin{example} \label{Ex_Circle_HomflyPt}
Let us denote by \(\zeroweight\) the element \((0,\dotsc,0,\beta-0,\dotsc,\beta-0) \in \lattice\). We have:
\begin{equation}\label{eq:5}
[E_k,F_k]\onell{\zeroweight}=[0-\beta+0] \onell{\zeroweight},
\end{equation}
and since \(E_kF_k\onell{\zeroweight}=0\) we get 
\begin{equation}\label{eq:4}
F_kE_k\onell{\zeroweight}= \frac{q^{\beta}-q^{-\beta}}{q-q^{-1}}. 
\end{equation}
In the case \(\beta=m>0\), this formula reminds of the value on the unknot for the \(\mathfrak{sl}_m \) link invariant, and in the case \(\beta\) generic for the Homfly-Pt polynomial \cite{Homfly,PT}.
\end{example}

It is easy to check that we have a symmetry in our definition of \(\UzklB\), which will correspond diagrammatically to a rotation of \(180\) degrees:
\begin{lemma} \label{prop:antiinclusion}
There is an anti-isomorphism \(\UzklB\mapsto \UzmmB{l+k}\) given by
\begin{equation}
\begin{gathered}
\onel \mapsto \onell{(\beta-\lambda_{k+l},\dots,\beta-\lambda_{k+1},\beta-\lambda_{k}, \dots, \beta-\lambda_{1})}  \\
E_i \mapsto F_{k+l-i},\qquad F_i\mapsto E_{k+l-i}.
\end{gathered}\label{eq:3}
\end{equation}
\end{lemma}

By precomposing this isomorphism with the obvious map \(\Uzk\mapto \UzklB\) which sends the generators of \(\Uzk\) to the first \(k-1\) generators of \(\UzklB\), we also get a map \(\Uzl \mapto \UzklB\) hitting the last \(l-1\) generators.

%
\section{Homfly-Pt polynomial from the doubled Schur algebra} 
%

The Homfly-Pt polynomial \cite{Homfly,PT} is a two-variable generalization of both the \(\sln\) polynomials and the Alexander polynomial. Although it appears as some generic version of the \(\sln\) polynomials, or maybe some limit version for \(n \rightarrow \infty\), there is no convenient description of it as an element of the intertwining algebra of \(U_q(\slnn{\infty})\) or some suitable quotient of it. Our goal here is, by passing to the Howe dual of the quantum group, to obtain a unified definition of the Homfly-Pt polynomials together with its specializations in a quantum group setting.

%
\subsection{A glimpse of diagrammatic calculus: ladder webs}
%

Based on the usual \(\sln\) case \cite{CKM}, it will be convenient to perform  some computations diagrammatically. We will assign to the generators of our category elementary \emph{ladder webs}, that is, trivalent graphs with labeled edges and some rigid properties, as follows:
\begin{gather*}
\onel \mapsto 
\begin{tikzpicture}[xscale=.75,yscale=1.5,anchorzero]
\draw [thick,->] (0,-.2) node[below] {\tiny \(\lambda_1\)} -- (0,.2);
\draw [thick,->] (.5,-.2) node[below] {\tiny \(\lambda_2\)} -- (.5,.2);
\node at (1,0) {\(\cdots\)};
\draw [thick,->] (1.5,-.2) node[below] {\tiny \(\lambda_{k+l}\)} -- (1.5,.2);
\end{tikzpicture}\\
\onell{\lambda+\alpha_i}E_i\onel \mapsto
\begin{tikzpicture}[scale=.75,baseline={([yshift=-0.5ex]0,0.375)}]
\draw [thick,->] (0,0) -- (0,1);
\node at (.5,.5) {\(\cdots\)};
\draw [thick,->] (1,0) -- (1,1);
\draw [thick,->] (1.75,0) -- (1.75,1);
\node at (2.25,.5) {\(\cdots\)};
\draw [thick,->] (2.75,0) -- (2.75,1);
\draw [directed=.5, thick] (1.75,.3) -- (1,.7);
\node at (0,-.3) {\tiny \(\lambda_1\)};
\node at (1,-.3) {\tiny \(\lambda_i\)};
\node at (1.75,-.3) {\tiny\(\lambda_{i+1}\)};
\node at (2.75,-.3) {\tiny\(\lambda_{k+l}\)};
\end{tikzpicture}\qquad
\onell{\lambda-\alpha_i}F_i\onel \mapsto
\begin{tikzpicture}[scale=.75,baseline={([yshift=-0.5ex]0,0.375)}]
\draw [thick,->] (0,0) -- (0,1);
\node at (.5,.5) {\(\cdots\)};
\draw [thick,->] (1,0) -- (1,1);
\draw [thick,->] (1.75,0) -- (1.75,1);
\node at (2.25,.5) {\(\cdots\)};
\draw [thick,->] (2.75,0) -- (2.75,1);
\draw [directed=.5, thick] (1,.3) -- (1.75,.7);
\node at (0,-.3) {\tiny \(\lambda_1\)};
\node at (1,-.3) {\tiny \(\lambda_i\)};
\node at (1.75,-.3) {\tiny \(\lambda_{i+1}\)};
\node at (2.75,-.3) {\tiny \(\lambda_{k+l}\)};
\end{tikzpicture}
\end{gather*}
Diagrammatic versions of the relations of Definition~\ref{defn_UzklB} hold for these diagrams.
We will often omit the edges labeled \(0\), or depict them with dots. We will also sometimes orient edges labeled with \(\beta-\N\) downward, and also depict the \(\beta-0\) edges dotted.
Although these graphs have an interesting topological behavior, we will here mostly use them as a shorthand for actual computations in the (doubled) Schur algebra. A result of \cite{CKM,QR} we will however use with moderation is that isotopies can be performed as long as we stay only on the left side of the picture, corresponding to the \(k\)-th first uprights. Then Proposition~\ref{prop:antiinclusion} implies that the same holds true for the last \(l\)-th uprights, leaving the problem at the interface between the \(\N\) and the \(\beta-\N\) labels.

Another trick we'll use from time to time, and which diagrammatically is extremely transparent, is to inject \(\UzklB\) into \(\UzmmB{(k+r)+(l+s)}\), sending a weight \(\lambda\)  to \((0,\dots,0,\lambda_1,\dots,\lambda_{k+l},\beta-0,\dots,\beta-0)\). This allows to create more space to perform isotopies when needed.

%
\subsection{A link invariant}
%

Let \(\boldsymbol{b}\) be a braid with \(k\) strands; then one can assign to it an element of \(\Uzk\) using the following local rule (see \cite{Lus4}):
\begin{align}
\begin{tikzpicture}[scale=.7,anchorbase]
\draw [thick,->] (0,0) -- (1,1);
\draw [thick] (1,0) -- (.6,.4);
\draw [thick,->] (.4,.6) -- (0,1);
\end{tikzpicture}
& \mapsto q^{-1}\onell{(1,1)}-EF\onell{(1,1)}\\[3pt]
\begin{tikzpicture}[scale=.7,anchorbase]
\draw [thick,->] (1,0) -- (0,1);
\draw [thick] (0,0) -- (.4,.4);
\draw [thick,->] (.6,.6) -- (1,1);
\end{tikzpicture}
& \mapsto -EF\onell{(1,1)}+q\onell{(1,1)}
\end{align}
Note that above, we focused on a \(2\)-strand braid. If the \(i\)-th and \((i+1)\)-st strands are involved, one  uses the Chevalley generators \(E_i\) and \(F_i\). As one reads the braid from bottom to top, one composes the quantum group expression from right to left (as function composition).

This element of \(\Uzk\) can be closed in \(\UzmmB{k+k}\) as follows:
\begin{equation}\label{eq:19}
\begin{tikzpicture}[yscale=.45,xscale=.7,anchorbase]
\draw (-1,-0.75) rectangle (1,0.75);
\node at (0,0) {\(\boldsymbol{b}\)};
\draw [thick,directed=.5] (.5,0.75) .. controls ++(0,1) and ++(0,1) .. (1.5,0.75) -- ++(0,-1.5) .. controls ++(0,-1) and ++(0,-1) .. (.5,-0.75);
\draw [thick,directed=.5] (0,0.75) .. controls ++(0,1.5) and ++(0,1.5) .. (2,0.75) -- ++(0,-1.5) .. controls ++(0,-1.5) and ++(0,-1.5) .. (0,-0.75);
\draw [thick,directed=.5] (-.5,0.75) .. controls ++(0,2) and ++(0,2) .. (2.5,0.75) -- ++(0,-1.5) .. controls ++(0,-2) and ++(0,-2) .. (-0.5,-0.75);
\end{tikzpicture}
\end{equation}
where the cups and caps are realized using the following elementary pattern:
\begin{equation}\label{eq:15}
\begin{tikzpicture}[scale=.75,anchorbase]
\draw [thick,directed=1] (0,0) -- (0,1.5) -- (1,1.5) -- (1,2); 
\draw [thick, directed=.55] (1,0) -- (1,.5) -- (2,.5) -- (2,0);
\draw [thick, directed=1] (2,2) -- (2,1.5) -- (3,1.5) -- (3,0);
\draw [dotted, directed=1] (0,1.5) -- (0,2);
\draw [dotted, directed=.55] (1,.5) -- (1,1.5);
\draw [dotted, directed=.55] (2,1.5) -- (2,.5);
\draw [dotted, directed=.55] (3,2) -- (3,1.5);
\node at (0,-.3) {\tiny \(1\)};
\node at (1,-.3) {\tiny \(1\)};
\node at (2,-.3) {\tiny \(\beta-1\)};
\node at (3,-.3) {\tiny \(\beta-1\)};
\end{tikzpicture}
\end{equation}

If \(\boldsymbol{b}\) is a braid, let us denote by  \(P(\boldsymbol{b})\) the associated element in \(\UzmmB{k+k}\) that we obtain through this process. We can now state our first result:

\begin{theorem}
The element \(P(\boldsymbol{b})\) is a framed invariant of the link \(\hat{\boldsymbol{b}}\), the closure of \(\boldsymbol{b}\).
\end{theorem}

The proof of the theorem consists in checking that \(P(\boldsymbol{b})\) is a braid invariant, and it is invariant under the Markov moves. We will do this in the following lemmas.

\begin{lemma}\label{P_invariance}
\(P(\boldsymbol{b})\) is a braid invariant.
\end{lemma}
\begin{proof}
This is a direct consequence of the fact that the braiding verifies the Yang-Baxter equation \cite{Lus4}. Actually, the result can also be proven by direct computation, checking that positive and negative crossings yield inverse elements, and the braiding relation is respected. For example, we have
\begin{align*}
\begin{tikzpicture}[yscale=.5,anchorbase]
\draw [thick,->] (.5,0) .. controls (.5,.5) and (0,.5) .. (0,1) .. controls (0,1.5) and (.5,1.5) .. (.5,2);
\draw [thick,->,cross line] (0,0) .. controls (0,.5) and (.5,.5) .. (.5,1) .. controls (.5,1.5) and (0,1.5) .. (0,2);
\end{tikzpicture}
&\mapsto
(-EF\onell{(1,1)}+q\onell{(1,1)})(q^{-1}\onell{(1,1)}-EF\onell{(1,1)})\\
&\qquad =\onell{(1,1)}-(q+q^{-1})EF\onell{(1,1)}+EFEF\onell{(1,1)} \\
&\qquad =\onell{(1,1)}-(q+q^{-1})EF\onell{(1,1)}+[2]EF\onell{(1,1)} \\
&\qquad =\onell{(1,1)}\qedhere
\end{align*}
\end{proof}

\begin{lemma} \label{lemma:Markov1}
\(P(\boldsymbol{b})\) is invariant under the first Markov's move.
\end{lemma}
\begin{proof}
The first Markov's move can be decomposed as a sequence of moves symmetric to the following one: 
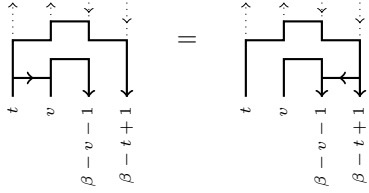
\[
\begin{tikzpicture}[scale=.5,baseline={([yshift=-0.5ex]0,0.75)}]
\draw [thick, directed=1] (0,0)  -- (0,1.5) -- (1,1.5) -- (1,2) -- (2,2) -- (2,1.5) -- (3,1.5) -- (3,0);
\draw [thick, directed=1] (1,0) -- (1,1) -- (2,1) -- (2,0);
\draw [thick, directed=.55] (0,.5) -- (1,.5);
\draw [dotted, directed=1] (0,1.5) -- (0,2.5);
\draw [dotted, directed=1] (1,2) -- (1,2.5);
\draw [dotted, directed=.55] (2,2.5) -- (2,2);
\draw [dotted, directed=.55] (3,2.5) -- (3,1.5);
\node [below] at (0,0) {\rotatebox{90}{\tiny \(t\)}};
\node [below] at (1,0) {\rotatebox{90}{\tiny \(v\)}};
\node [below] at (2,0) {\rotatebox{90}{\tiny \(\beta-v-1\)}};
\node [below] at (3,0) {\rotatebox{90}{\tiny \(\beta-t+1\)}};
\end{tikzpicture}
\quad
=
\quad
\begin{tikzpicture}[scale=.5,baseline={([yshift=-0.5ex]0,0.75)}]
\draw [thick, directed=1] (0,0) -- (0,1.5) -- (1,1.5) -- (1,2) -- (2,2) -- (2,1.5) -- (3,1.5) -- (3,0);
\draw [thick, directed=1] (1,0) -- (1,1) -- (2,1) -- (2,0);
\draw [thick, directed=.55] (3,.5) -- (2,.5);
\draw [dotted, directed=1] (0,1.5) -- (0,2.5);
\draw [dotted, directed=1] (1,2) -- (1,2.5);
\draw [dotted, directed=.55] (2,2.5) -- (2,2);
\draw [dotted, directed=.55] (3,2.5) -- (3,1.5);
\node [below] at (0,0) {\rotatebox{90}{\tiny \(t\)}};
\node [below] at (1,0) {\rotatebox{90}{\tiny \(v\)}};
\node [below] at (2,0) {\rotatebox{90}{\tiny \(\beta-v-1\)}};
\node [below] at (3,0) {\rotatebox{90}{\tiny \(\beta-t+1\)}};
\end{tikzpicture}
\]

This move decomposes as follows:

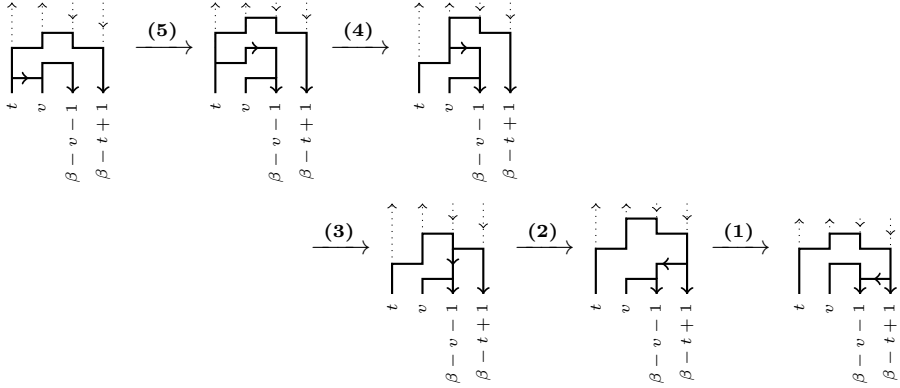
\begin{multline*}
\begin{tikzpicture}[scale=.4,baseline={([yshift=-0.5ex]0,0.6)}]
\draw [thick, directed=1] (0,0) -- (0,1.5) -- (1,1.5) -- (1,2) -- (2,2) -- (2,1.5) -- (3,1.5) -- (3,0);
\draw [thick, directed=1] (1,0) -- (1,1) -- (2,1) -- (2,0);
\draw [thick, directed=.55] (0,.5) -- (1,.5);
\draw [dotted, directed=1] (0,1.5) -- (0,3);
\draw [dotted, directed=1] (1,2) -- (1,3);
\draw [dotted, directed=.55] (2,3) -- (2,2);
\draw [dotted, directed=.55] (3,3) -- (3,1.5);
\node [below] at (0,0) {\rotatebox{90}{\tiny \(t\)}};
\node [below] at (1,0) {\rotatebox{90}{\tiny \(v\)}};
\node [below] at (2,0) {\rotatebox{90}{\tiny \(\beta-v-1\)}};
\node [below] at (3,0) {\rotatebox{90}{\tiny \(\beta-t+1\)}};
\end{tikzpicture}
\;
\xrightarrow{\;\boldsymbol{(5)}\;}
\begin{tikzpicture}[scale=.4,baseline={([yshift=-0.5ex]0,0.6)}]
\draw [thick, directed=1] (0,0) -- (0,2) -- (1,2) -- (1,2.5) -- (2,2.5) -- (2,2) -- (3,2) -- (3,0);
\draw [thick, directed=1] (1,0) -- (1,.5) -- (2,.5) -- (2,0);
\draw [thick, directed=.55] (0,1) -- (1,1) -- (1,1.5) -- (2,1.5) -- (2,.5);
\draw [dotted, directed=1] (0,1.5) -- (0,3);
\draw [dotted, directed=1] (1,2) -- (1,3);
\draw [dotted, directed=.55] (2,3) -- (2,2.5);
\draw [dotted, directed=.55] (3,3) -- (3,2);
\node [below] at (0,0) {\rotatebox{90}{\tiny \(t\)}};
\node [below] at (1,0) {\rotatebox{90}{\tiny \(v\)}};
\node [below] at (2,0) {\rotatebox{90}{\tiny \(\beta-v-1\)}};
\node [below] at (3,0) {\rotatebox{90}{\tiny \(\beta-t+1\)}};
\end{tikzpicture}
\xrightarrow{\;\boldsymbol{(4)}\;}
\;
\begin{tikzpicture}[scale=.4,baseline={([yshift=-0.5ex]0,0.6)}]
\draw [thick, directed=1] (0,0) -- (0,1) -- (1,1) -- (1,2.5) -- (2,2.5) -- (2,2) -- (3,2) -- (3,0);
\draw [thick, directed=1] (1,0) -- (1,.5) -- (2,.5) -- (2,0);
\draw [thick, directed=.3] (1,1.5) -- (2,1.5) -- (2,.5);
\draw [dotted, directed=1] (0,1) -- (0,3);
\draw [dotted, directed=1] (1,2.5) -- (1,3);
\draw [dotted, directed=.55] (2,3) -- (2,2.5);
\draw [dotted, directed=.55] (3,3) -- (3,2);
\node [below] at (0,0) {\rotatebox{90}{\tiny \(t\)}};
\node [below] at (1,0) {\rotatebox{90}{\tiny \(v\)}};
\node [below] at (2,0) {\rotatebox{90}{\tiny \(\beta-v-1\)}};
\node [below] at (3,0) {\rotatebox{90}{\tiny \(\beta-t+1\)}};
\end{tikzpicture}
\\
\xrightarrow{\;\boldsymbol{(3)}\;}
\begin{tikzpicture}[scale=.4,baseline={([yshift=-0.5ex]0,0.6)}]
\draw [thick, directed=1] (0,0) -- (0,1) -- (1,1) -- (1,2) -- (2,2) -- (2,1.5) -- (3,1.5) -- (3,0);
\draw [thick, directed=1] (1,0) -- (1,.5) -- (2,.5) -- (2,0);
\draw [thick, directed=.55] (2,1.5) -- (2,.5);
\draw [dotted, directed=1] (0,1) -- (0,3);
\draw [dotted, directed=1] (1,2) -- (1,3);
\draw [dotted, directed=.55] (2,3) -- (2,2);
\draw [dotted, directed=.55] (3,3) -- (3,1.5);
\node [below] at (0,0) {\rotatebox{90}{\tiny \(t\)}};
\node [below] at (1,0) {\rotatebox{90}{\tiny \(v\)}};
\node [below] at (2,0) {\rotatebox{90}{\tiny \(\beta-v-1\)}};
\node [below] at (3,0) {\rotatebox{90}{\tiny \(\beta-t+1\)}};
\end{tikzpicture}
\;
\xrightarrow{\;\boldsymbol{(2)}\;}
\begin{tikzpicture}[scale=.4,baseline={([yshift=-0.5ex]0,0.6)}]
\draw [thick, directed=1] (0,0) -- (0,1.5) -- (1,1.5) -- (1,2.5) -- (2,2.5) -- (2,2) -- (3,2) -- (3,0);
\draw [thick, directed=1] (1,0) -- (1,.5) -- (2,.5) -- (2,0);
\draw [thick, directed=.55] (3,1) -- (2,1) -- (2,.5);
\draw [dotted, directed=1] (0,1.5) -- (0,3);
\draw [dotted, directed=1] (1,2.5) -- (1,3);
\draw [dotted, directed=.55] (2,3) -- (2,2.5);
\draw [dotted, directed=.55] (3,3) -- (3,2);
\node [below] at (0,0) {\rotatebox{90}{\tiny \(t\)}};
\node [below] at (1,0) {\rotatebox{90}{\tiny \(v\)}};
\node [below] at (2,0) {\rotatebox{90}{\tiny \(\beta-v-1\)}};
\node [below] at (3,0) {\rotatebox{90}{\tiny \(\beta-t+1\)}};
\end{tikzpicture}
 \xrightarrow{\;\boldsymbol{(1)}\;}
\;
\begin{tikzpicture}[scale=.4,baseline={([yshift=-0.5ex]0,0.6)}]
\draw [thick, directed=1] (0,0) -- (0,1.5) -- (1,1.5) -- (1,2) -- (2,2) -- (2,1.5) -- (3,1.5) -- (3,0);
\draw [thick, directed=1] (1,0) -- (1,1) -- (2,1) -- (2,0);
\draw [thick, directed=.55] (3,.5) -- (2,.5);
\draw [dotted, directed=1] (0,1.5) -- (0,2.5);
\draw [dotted, directed=1] (1,2) -- (1,2.5);
\draw [dotted, directed=.55] (2,2.5) -- (2,2);
\draw [dotted, directed=.55] (3,2.5) -- (3,1.5);
\node [below] at (0,0) {\rotatebox{90}{\tiny \(t\)}};
\node [below] at (1,0) {\rotatebox{90}{\tiny \(v\)}};
\node [below] at (2,0) {\rotatebox{90}{\tiny \(\beta-v-1\)}};
\node [below] at (3,0) {\rotatebox{90}{\tiny \(\beta-t+1\)}};
\end{tikzpicture}
\end{multline*}

Move \(\boldsymbol{(1)}\) is a consequence of \(E_3F_2^{(v)}=F_2^{(v)}E_3\), and \(\boldsymbol{(2)}\) is a consequence of \eqref{eq:higherEF}. Move
\(\boldsymbol{(3)}\)  follows from the following local move:
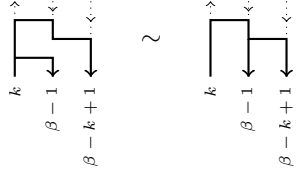
\[
\begin{tikzpicture}[scale=.5,baseline={([yshift=-0.5ex]0,0.5)}]
\draw [thick, directed=1] (0,0) -- (0,1.5) -- (1,1.5) -- (1,1) -- (2,1) -- (2,0);
\draw [thick, directed=1] (0,.5) -- (1,.5) -- (1,0);
\draw [dotted, directed=1] (0,1.5) -- (0,2);
\draw [dotted, directed=.55] (1,2) -- (1,1.5);
\draw [dotted, directed=.55] (2,2) -- (2,1);
\node [below] at (0,0) {\rotatebox{90}{\tiny \(k\)}};
\node [below] at (1,0) {\rotatebox{90}{\tiny \(\beta-1\)}};
\node [below] at (2,0) {\rotatebox{90}{\tiny \(\beta-k+1\)}};
\end{tikzpicture}
\quad
\sim
\quad
\begin{tikzpicture}[scale=.5,baseline={([yshift=-0.5ex]0,0.5)}]
\draw [thick, directed=1] (0,0) -- (0,1.5) -- (1,1.5) -- (1,1) -- (2,1) -- (2,0);
\draw [thick, directed=1] (1,1) -- (1,0);
\draw [dotted, directed=1] (0,1.5) -- (0,2);
\draw [dotted, directed=.55] (1,2) -- (1,1.5);
\draw [dotted, directed=.55] (2,2) -- (2,1);
\node [below] at (0,0) {\rotatebox{90}{\tiny \(k\)}};
\node [below] at (1,0) {\rotatebox{90}{\tiny \(\beta-1\)}};
\node [below] at (2,0) {\rotatebox{90}{\tiny \(\beta-k+1\)}};
\end{tikzpicture}
\]
This corresponds to proving
\[F_1^{(k-1)}F_2^{(k-1)}F_1\onell{(k,\beta-1,\beta-k+1)}=F_1^{(k)}F_2^{(k-1)}\onell{(k,\beta-1,\beta-k+1)}.\]
From \eqref{eq:1} we get
\[
F_1^{(k)}F_2^{(k-1)}=\sum_{r+s=k}(-1)^{s+1}F_1^{(r)}F_2^{(k-1)}F_1^{(s)}.
\]
Unless \(s\leq 1\), the corresponding term acting on \(\onell{(k,\beta-1,\beta-k+1)}\) is zero, and we thus get as desired:
\[
F_1^{(k)}F_2^{(k-1)}\onell{(k,\beta-1,\beta-k+1)} =F_1^{(k-1)}F_2^{(k-1)}F_1\onell{(k,\beta-1,\beta-k+1)}.
\]

Move \(\boldsymbol{(4)}\) locally reduces to:
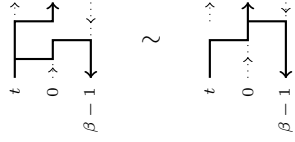
\[
\begin{tikzpicture}[scale=.5,baseline={([yshift=-0.5ex]0,0.5)}]
\draw [thick, directed=1] (0,0) -- (0,1.5) -- (1,1.5) -- (1,2);
\draw [thick, directed=1] (0,.5) -- (1,.5) -- (1,1) -- (2,1) -- (2,0);
\draw [dotted, directed=1] (0,1.5) -- (0,2);
\draw [dotted, directed=.55] (1,0) -- (1,.5);
\draw [dotted, directed=.55] (2,2) -- (2,1);
\node [below] at (0,0) {\rotatebox{90}{\tiny \(t\)}};
\node [below] at (1,0) {\rotatebox{90}{\tiny \(0\)}};
\node [below] at (2,0) {\rotatebox{90}{\tiny \(\beta-1\)}};
\end{tikzpicture}
\quad
\sim
\quad
\begin{tikzpicture}[scale=.5,baseline={([yshift=-0.5ex]0,0.5)}]
\draw [thick, directed=1] (0,0) -- (0,1) -- (1,1) -- (1,2);
\draw [thick, directed=1] (1,1.5) -- (2,1.5) -- (2,0);
\draw [dotted, directed=1] (0,1.5) -- (0,2);
\draw [dotted, directed=.55] (1,0) -- (1,1);
\draw [dotted, directed=.55] (2,2) -- (2,1.5);
\node [below] at (0,0) {\rotatebox{90}{\tiny \(t\)}};
\node [below] at (1,0) {\rotatebox{90}{\tiny \(0\)}};
\node [below] at (2,0) {\rotatebox{90}{\tiny \(\beta-1\)}};
\end{tikzpicture}
\]
This follows from
\begin{equation*}
\begin{aligned}
F_1^{(t-1)}F_2F_1\onell{(t,0,\beta-1)} & = F_1^{(t-1)}F_2F_1E_1^{(t)}F_1^{(t)}\onell{(t,0,\beta-1)} \\
& = F_1^{(t-1)}F_2E_1^{(t-1)}F_1^{(t)}\onell{(t,0,\beta-1)}  \\
& = F_1^{(t-1)}E_1^{(t-1)}F_2F_1^{(t)}\onell{(t,0,\beta-1)}=F_2F_1^{(t)}\onell{(t,0,\beta-1)}.
\end{aligned}
\end{equation*}

Finally, move \(\boldsymbol{(5)}\) reduces to:
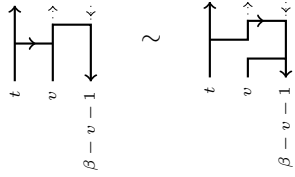
\[
\begin{tikzpicture}[scale=.5,baseline={([yshift=-0.5ex]0,0.55)}]
\draw [thick, directed=1] (0,0) -- (0,2);
\draw [thick, directed=1] (1,0) -- (1,1.5) -- (2,1.5) -- (2,0);
\draw [thick, directed=.55] (0,1) -- (1,1);
\draw [dotted, directed=1] (1,1.5) -- (1,2);
\draw [dotted, directed=.55] (2,2) -- (2,1.5);
\node [below] at (0,0) {\rotatebox{90}{\tiny \(t\)}};
\node [below] at (1,0) {\rotatebox{90}{\tiny \(v\)}};
\node [below] at (2,0) {\rotatebox{90}{\tiny \(\beta-v-1\)}};
\end{tikzpicture}
\quad
\sim
\quad
\begin{tikzpicture}[scale=.5,baseline={([yshift=-0.5ex]0,0.5)}]
\draw [thick, directed=1] (0,0) -- (0,2);
\draw [thick, directed=1] (1,0) -- (1,.5) -- (2,.5) -- (2,0);
\draw [thick, directed=.55] (0,1) -- (1,1) -- (1,1.5) -- (2,1.5) -- (2,.5);
\draw [dotted, directed=1] (1,1.5) -- (1,2);
\draw [dotted, directed=.55] (2,2) -- (2,1.5);
\node [below] at (0,0) {\rotatebox{90}{\tiny \(t\)}};
\node [below] at (1,0) {\rotatebox{90}{\tiny \(v\)}};
\node [below] at (2,0) {\rotatebox{90}{\tiny \(\beta-v-1\)}};
\end{tikzpicture}
\]
That is, we want to prove
\begin{equation}
F_2^{(v+1)}F_1\onell{(t,v,\beta-v-1)}=F_2F_1F_2^{(v)}\onell{(t,v,\beta-v-1)}\label{eq:6}
\end{equation}
We proceed by induction. The case \(v=0\) is obvious. For the general case, we use again a version of the higher quantum Serre relation \eqref{eq:1}:
\begin{equation}
F_2^{(v+1)}F_1=\sum_{r=1}^{v+1} (-1)^{r-1}q^{-(v-1)r}F_2^{(v+1-r)}F_1F_2^{(r)}.\label{eq:7}
\end{equation}
Since \(F_2^{(v+1)}\onell{(t,v,\beta-v-1)}=0\), the sum becomes:
\begin{equation}
F_2^{(v+1)}F_1=\sum_{r=1}^{v} (-1)^{r-1}q^{-(v-1)r}F_2^{(v+1-r)}F_1F_2^{(r)}.\label{eq:8}
\end{equation}
We use the induction hypothesis:
\begin{equation}
F_2^{(v+1-r)}F_1\onell{(t,v-r,\beta-v+r-1)}=F_2F_1F_2^{(v-r)}\onell{(t,v-r,\beta-v+r-1)}\label{eq:9}
\end{equation}
to obtain:
\begin{equation}
\begin{aligned}
F_2^{(v+1)}F_1\onell{(t,v,\beta-v-1)}&=\sum_{r=1}^v(-1)^{r-1}q^{-(v-1)r}F_2F_1F_2^{(v-r)}F_2^{(r)}\onell{(t,v,\beta-v-1)}  \\
&=\sum_{r=1}^{v}(-1)^{r-1}q^{-(v-1)r}\qbin{v}{r} F_2F_1F_2^{(v)} \onell{(t,v,\beta-v-1)}  \\
&= \left(1-\sum_{r=0}^v(-1)^rq^{-(v-1)r} \qbin{v}{r} \right) F_2F_1F_2^{(v)}\onell{(t,v,\beta-v-1)}. \label{Move5_keyEq}
\end{aligned}
\end{equation}
We then prove \(\sum_{r=0}^{v} (-1)^r q^{-(v-1)r}\qbin{v}{r} =0\), also by induction on \(v\). The first step is trivial, and the induction step goes as follows:
\begin{align*}
\sum_{r=0}^{v+1} (-1)^r q^{-vr} \qbin{v+1}{r} &=\sum_{r=0}^{v+1}(-1)^r q^{-(v-1)r} \qbin{v}{r}
+ \sum_{r=0}^{v+1} (-1)^r q^{-vr-v+r} \qbin{v}{r-1} \\
& =\sum_{r=0}^v(-1)^rq^{-(v-1)r} \qbin{v}{r}
 + q^{-v}\sum_{r=1}^{v+1}(-1)^rq^{-(v-1)r} \qbin{v}{r} \\
 &= 0-q^{-v}\sum_{r=0}^v(-1)^rq^{-(v-1)r}q^{-(v-1)} \qbin{v}{r} \\
 &= 0-q^{2-v+1}\sum_{r=0}^v(-1)^rq^{-(v-1)r} \qbin{v}{r} =0.
\end{align*}
For the computation, we used that the usual binomial induction formula lifts in the quantum case as follows:
\begin{equation}
\qbin{n+1}{r}
=q^r
\qbin{n}{r}
+q^{-n-1+r}
\qbin{n}{r-1}.\label{eq:10}
\end{equation}
This implies \eqref{eq:6}, and concludes the proof.
\end{proof}

\begin{lemma}
A positive second Markov move multiplies the braid invariant by \(q^{-\beta}\), and a negative one by \(q^{\beta}\).
\end{lemma}

\begin{proof}
The claim comes down to computing:
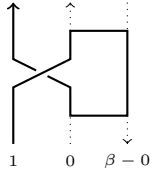
\[
\begin{tikzpicture}[scale=.75,anchorbase]
\draw[thick] (0,0) -- (0,1) -- (1,1.5) -- (1,2) -- (2,2) -- (2,.5) -- (1,.5) -- (1,1) -- (.6,1.2);
\draw [thick, ->] (.4,1.3) -- (0,1.5) -- (0,2.5);
\draw [dotted] (1,0) -- (1,.5);
\draw [dotted,->] (1,2) -- (1,2.5);
\draw [dotted] (2,2.5) -- (2,2);
\draw [dotted,->] (2,.5) -- (2,0);
\node at (0,-.3) {\tiny \(1\)};
\node at (1,-.3) {\tiny \(0\)};
\node at (2,-.3) {\tiny \(\beta-0\)};
\end{tikzpicture}
\]
which corresponds to
\begin{equation}
q^{-1}F_2E_2\onell{(1,0,\beta-0)}-F_2E_1F_1E_2\onell{(1,0,\beta-0)}.\label{eq:11}
\end{equation}
Noting that
\begin{equation}
F_2E_2\onell{(1,0,\beta-0)}=[\beta]\onell{(1,0,\beta-0)}\label{eq:12}
\end{equation}
and
\begin{multline}\label{eq:13}
F_2E_1F_1E_2\onell{(1,0,\beta-0)}=E_1F_2E_2\onell{(0,1,\beta-0)}F_1\onell{(1,0,\beta-0)}\\=[\beta-1]E_1F_1\onell{(1,0,\beta-0)}  =[\beta-1]\onell{(1,0,\beta-0)},
\end{multline}
we obtain
\begin{multline}
q^{-1}F_2E_2\onell{(1,0,\beta-0)}-F_2E_1F_1E_2\onell{(1,0,\beta-0)}\\
=q^{-1}[\beta]\onell{(1,0,\beta-0)}-[\beta-1]\onell{(1,0,\beta-0)} = q^{-\beta}\onell{(1,0,\beta-0)}. 
\end{multline}
The computation for the negative case is similar.
\end{proof}

%
\subsection{The Homfly-Pt polynomial}
%

We now want to ensure that the invariant we get is a polynomial, before checking it is actually the Homfly-Pt polynomial.
The crucial point is that all terms in the smoothing of \(P({\boldsymbol{b}})\) decompose into concentric circles:

\begin{lemma} \label{lemma_circles_HomflyPt}
The element \(P(\boldsymbol{b})\) can be written as linear combination of monomials in the generators of \(\UzmmB{k+k}\), such that each term appearing corresponds diagrammatically to some concentric circles.
\end{lemma}

\begin{proof}\footnote{See \cite{QR2} for more details and a lift of this proof at the categorified level in a different context.}
By construction, the element  \(P(\boldsymbol{b})\) is written, as in the following picture, as linear combination of monomials of \(\UzklB\) of the type \(C \boldsymbol{w} C'\), where \(C\) is a cap, \(C'\) is a cup and \(\boldsymbol{w}\) is a monomial in the first \(k-1\) generators:
\begin{equation}
  \label{eq:20}
  \begin{tikzpicture}[yscale=.45,xscale=.7,anchorbase]
\draw (-1,-0.75) rectangle (1,0.75);
\node at (0,0) {\(\boldsymbol{w}\)};
\draw [thick,directed=.5] (.5,0.75) .. controls ++(0,1) and ++(0,1) .. (1.5,0.75) -- ++(0,-1.5) .. controls ++(0,-1) and ++(0,-1) .. (.5,-0.75);
\draw [thick,directed=.5] (0,0.75) .. controls ++(0,1.5) and ++(0,1.5) .. (2,0.75) -- ++(0,-1.5) .. controls ++(0,-1.5) and ++(0,-1.5) .. (0,-0.75);
\draw [thick,directed=.5] (-.5,0.75) .. controls ++(0,2) and ++(0,2) .. (2.5,0.75) -- ++(0,-1.5) .. controls ++(0,-2) and ++(0,-2) .. (-0.5,-0.75);
\end{tikzpicture}
\end{equation}
Notice that by the first Makov's move  (Lemma~\ref{lemma:Markov1}), or rather by its proof, we can slide around parts of \(\mathfrak{m}\) from the top to the bottom.
 We want to prove that any such element  can be written as a linear combination of monomials of \(\UzklB\) that correspond on the diagrammatic side to some concentric circles (possibly labeled by number higher than \(1\)). 

Let us  focus on the leftmost upright of the ladder \(\boldsymbol{w}\): if this is just a vertical line with no singularities, then its closure is already a circle and we are done by induction. Otherwise there are some \(E_1\)'s and \(F_1\)'s appearing in \(\boldsymbol{w}\) (creating ``singularities''). The idea of the proof is to perform a series of moves which will increase the labeling of this leftmost upright.
First, using the relations
\begin{equation}
\label{eq:22}
\begin{tikzpicture}[scale=.75,anchorbase]
\draw [thick,->] (0,0) -- (0,2);
\draw [thick,->] (1,0) -- (1,2);
\draw [thick,directed=.55] (0,.5) -- (1,.5);
\draw [thick,directed=.55] (0,1.5) -- (1,1.5);
\node at (.5,.2) {\tiny \(a\)}; 
\node at (.5,1.2) {\tiny \(b\)}; 
\end{tikzpicture}
\;=
\qbin{a+b}{a}\;
\begin{tikzpicture}[scale=.75,anchorbase]
\draw [thick,->] (0,0) -- (0,2);
\draw [thick,->] (1,0) -- (1,2);
\draw [thick,directed=.55] (0,1) -- (1,1);
\node at (.5,.2) {\tiny \(a\)}; 
\end{tikzpicture}
\qquad\text{and}\qquad \begin{tikzpicture}[scale=.75,anchorbase]
\draw [thick,->] (0,0) -- (0,2);
\draw [thick,->] (1,0) -- (1,2);
\draw [thick,rdirected=.55] (0,.5) -- (1,.5);
\draw [thick,rdirected=.55] (0,1.5) -- (1,1.5);
\node at (.5,.2) {\tiny \(a\)}; 
\node at (.5,1.2) {\tiny \(b\)}; 
\end{tikzpicture}
\;=
\qbin{a+b}{a}\;
\begin{tikzpicture}[scale=.75,anchorbase]
\draw [thick,->] (0,0) -- (0,2);
\draw [thick,->] (1,0) -- (1,2);
\draw [thick,rdirected=.55] (0,1) -- (1,1);
\node at (.5,.7) {\tiny \(a+b\)}; 
\end{tikzpicture}
\end{equation}
we can assume that there are no successive \(E_i\) or \(F_i\).

Recall that we can  slide far-away sequences, and adjacent \(E_i\)'s and \(F_j\)'s as follows:
\begingroup
\tikzset{every picture/.style={xscale=1.5}}
\begin{equation}
\label{eq:23}
\begin{tikzpicture}[scale=.75,anchorbase]
\draw [thick, ->] (0,0) -- (0,1.5);
\draw [thick,->] (.5,0) -- (.5,1.5);
\draw [thick,->] (1,0) -- (1,1.5);
\draw [thick, directed=.5] (.5,.5) -- (0,.5);
\draw [thick, directed=.5] (.5,1) -- (1,1);
\end{tikzpicture}\;=\;
\begin{tikzpicture}[scale=.75,anchorbase]
\draw [thick, ->] (0,0) -- (0,1.5);
\draw [thick,->] (.5,0) -- (.5,1.5);
\draw [thick,->] (1,0) -- (1,1.5);
\draw [thick, directed=.5] (.5,1) -- (0,1);
\draw [thick, directed=.5] (.5,.5) -- (1,.5);
\end{tikzpicture}\qquad\text{and}\qquad\begin{tikzpicture}[scale=.75,anchorbase]
\draw [thick, ->] (0,0) -- (0,1.5);
\draw [thick,->] (.5,0) -- (.5,1.5);
\draw [thick,->] (1,0) -- (1,1.5);
\draw [thick, rdirected=.5] (.5,.5) -- (0,.5);
\draw [thick, rdirected=.5] (.5,1) -- (1,1);
\end{tikzpicture} \;= \;
\begin{tikzpicture}[scale=.75,anchorbase]
\draw [thick, ->] (0,0) -- (0,1.5);
\draw [thick,->] (.5,0) -- (.5,1.5);
\draw [thick,->] (1,0) -- (1,1.5);
\draw [thick, rdirected=.5] (.5,1) -- (0,1);
\draw [thick, rdirected=.5] (.5,.5) -- (1,.5);
\end{tikzpicture}
\end{equation}
\endgroup
We now want to increase the labeling of the  far-left ladder upright until it becomes a circle. For this, we use in all possible places in the far-left inter-uprights the following consequence of \eqref{eq:higherEF}:
\begin{equation}
\label{eq:21}
\begin{tikzpicture}[scale=.75,anchorbase]
\draw [thick, ->] (0,0) -- (0,2);
\draw [thick, ->] (1,0) -- (1,2);
\draw [thick,directed=.55] (0,.5) -- (1,.5);
\draw [thick,directed=.55] (1,1.5) -- (0,1.5);
\end{tikzpicture}\;
=\sum \text{\tiny (coeff)} \;
\begin{tikzpicture}[scale=.75,anchorbase]
\draw [thick, ->] (0,0) -- (0,2);
\draw [thick, ->] (1,0) -- (1,2);
\draw [thick,rdirected=.55] (0,.5) -- (1,.5);
\draw [thick,rdirected=.55] (1,1.5) -- (0,1.5);
\end{tikzpicture}
\end{equation}

This results in increasing the labeling of the left strand (or in deleting completely the ladder rungs). Of course it can happen that two successive \(E_1\) and \(F_1\) are separated by a sequence of adjacent \(E_i\)'s and \(F_i\)'s. Using \eqref{eq:21} and \eqref{eq:23}, this case actually reduces to a situation similar to the following one (we omit the coefficients in the depiction below): we want to slide \(E_1^{(a)}\) and \(F_1^{(f)}\), but we have a sequence \(E_1^{(a_1)}E_2^{(a_2)}\dotsb E_i^{(a_i)}F_i^{(b_i)}\dotsb F_2^{(b_2)}F_1^{(b_1)}\). In this case, we start applying \eqref{eq:21} to \(E_i^{(a_i)}F_i^{a_i}\), then use \eqref{eq:23} to take them away and iterate until we are able to perform the desired relation for \(E_1\) and \(F_1\), as in the picture
\begin{equation}
  \label{eq:24}
\begin{tikzpicture}[scale=.5,anchorbase]
\draw[thick,->] (0,0) -- (0,3.5);
\draw[thick,->] (1,0) -- (1,3.5);
\draw[thick,->] (2,0) -- (2,3.5);
\draw[thick,->] (3,0) -- (3,3.5);
\draw[thick,directed=.55] (0,.5) -- (1,.5);
\draw[thick,directed=.55] (1,1) -- (2,1);
\draw[thick,directed=.55] (2,1.5) -- (3,1.5);
\draw[thick,directed=.55] (3,2) -- (2,2);
\draw[thick,directed=.55] (2,2.5) -- (1,2.5);
\draw[thick,directed=.55] (1,3) -- (0,3);
\end{tikzpicture}
=\sum
\begin{tikzpicture}[scale=.5,anchorbase]
\draw[thick,->] (0,0) -- (0,3.5);
\draw[thick,->] (1,0) -- (1,3.5);
\draw[thick,->] (2,0) -- (2,3.5);
\draw[thick,->] (3,0) -- (3,3.5);
\draw[thick,directed=.55] (0,.5) -- (1,.5);
\draw[thick,directed=.55] (1,1) -- (2,1);
\draw[thick,rdirected=.55] (2,1.5) -- (3,1.5);
\draw[thick,rdirected=.55] (3,2) -- (2,2);
\draw[thick,directed=.55] (2,2.5) -- (1,2.5);
\draw[thick,directed=.55] (1,3) -- (0,3);
\end{tikzpicture}
=\sum
\begin{tikzpicture}[scale=.5,anchorbase]
\draw[thick,->] (0,0) -- (0,3.5);
\draw[thick,->] (1,0) -- (1,3.5);
\draw[thick,->] (2,0) -- (2,3.5);
\draw[thick,->] (3,0) -- (3,3.5);
\draw[thick,directed=.55] (3,.5) -- (2,.5);
\draw[thick,directed=.55] (0,1) -- (1,1);
\draw[thick,directed=.55] (1,1.5) -- (2,1.5);
\draw[thick,directed=.55] (2,2) -- (1,2);
\draw[thick,directed=.55] (1,2.5) -- (0,2.5);
\draw[thick,directed=.55] (2,3) -- (3,3);
\end{tikzpicture}
=\cdots=
\sum
\begin{tikzpicture}[scale=.5,anchorbase]
\draw[thick,->] (0,0) -- (0,3.5);
\draw[thick,->] (1,0) -- (1,3.5);
\draw[thick,->] (2,0) -- (2,3.5);
\draw[thick,->] (3,0) -- (3,3.5);
\draw[thick,rdirected=.55] (0,1.5) -- (1,1.5);
\draw[thick,rdirected=.55] (1,1) -- (2,1);
\draw[thick,rdirected=.55] (2,.5) -- (3,.5);
\draw[thick,rdirected=.55] (3,3) -- (2,3);
\draw[thick,rdirected=.55] (2,2.5) -- (1,2.5);
\draw[thick,rdirected=.55] (1,2) -- (0,2);
\end{tikzpicture}
\end{equation}

Once all of these moves have been performed, we are left with a web \(\boldsymbol{w}\) where all \(F_1\)'s are over the \(E_1\)'s. We can then slide around the part containing the \(F_1\)'s from the top of the web to the bottom, and iterate the process until there are no \(E_1\)'s and \(F_1\)'s left. Since we increase the label of the leftmost ladder upright, but the sum of the labels of the uprights remains constant, the process has to terminate: this  means that there are no \(E_1\)'s and \(F_1\)'s left and hence the leftmost upright is a circle. We can then pass to the second upright, and by induction  we are done.
\end{proof}

\begin{remark}
  Together with arguments from the proof of Lemma~\ref{lemma_OneDim} below, we obtain from Lemma~\ref{lemma_circles_HomflyPt} that the endomorphism algebra of the identity object is one-dimensional.
\end{remark}

Following \cite{CKM}, let us denote by \(\Uzmk\) the quotient of \(\Uzk\) by the ideal generated by weights with an entry greater than \(m\) (that is, if a morphism factors through  \(\onel\) and \(\lambda\) has an entry greater than \(m\), then we set this morphism equal to zero). The following lemma is easy to check:

\begin{lemma} \label{lemma:specialization}
If \(\beta=m\) then we have  a functor
\begin{equation}
\UzklB\mapto \Uznnmm{m}{k+l}.\label{eq:14}
\end{equation}
\end{lemma}

\begin{theorem} \label{thm:HomflyPt}
If \(\beta\) is generic, then \(P(\boldsymbol{b})\) is the Homfly-Pt polynomial of \(\hat{\boldsymbol{b}}\). If \(\beta = m \in \Z\), with \(m \geq 2\) then \(P(\boldsymbol{b})\) is the \(\mathfrak{sl}_m\) link invariant of \(\hat{\boldsymbol{b}}\).
\end{theorem}

\begin{proof}
First, using Lemma~\ref{lemma_circles_HomflyPt}, the associated invariant decomposes into a sum of elements associated to concentric circles. From Example~\ref{Ex_Circle_HomflyPt}, we know that a circle evaluates to a polynomial in \((q,q^{\beta})\), and evaluating all of the circles does give a polynomial. We want to check that it verifies the same skein and normalization relations as the Homfly-Pt polynomial (respectively, the \(\mathfrak{sl}_m\) link invariant), which implies the claim.

The skein relation can be checked as follows:
\begin{equation*}
\begin{tikzpicture}[scale=.75,anchorbase]
\draw [thick,->] (0,0) -- (1,1);
\draw [thick] (1,0) -- (.6,.4);
\draw [thick,->] (.4,.6) -- (0,1);
\end{tikzpicture}
-
\begin{tikzpicture}[scale=.75,anchorbase]
\draw [thick,->] (1,0) -- (0,1);
\draw [thick] (0,0) -- (.4,.4);
\draw [thick,->] (.6,.6) -- (1,1);
\end{tikzpicture}
\mapsto
q^{-1}\onell{(1,1)}-EF\onell{(1,1)}+EF\onell{(1,1)}-q\onell{(1,1)}
\mapsfrom (q^{-1}-q)\; 
\begin{tikzpicture}[scale=.75,anchorbase]
\draw [thick,->] (1,0) -- (1,1);
\draw [thick,->] (0,0) -- (0,1);
\end{tikzpicture}
\end{equation*}
and the unknot value is provided by Example~\ref{Ex_Circle_HomflyPt}.
\end{proof}

%
\section{The Alexander polynomial} 
%

We denote by \(\Uzerokl\) the \(\beta=0\) specialization of \(\UzklB\).
We have a similar decomposition as in the usual Schur algebra case:
\begin{equation}
\Uzerokl=\bigoplus_{N\in \Z}\U_q(\glnn{k+l})_{0}^N.\label{eq:16}
\end{equation}
In this decomposition, each subcategory has infinitely many objects. For example, \(\U_q(\glnn{1+1})_{0}^{0}\) has objects \((p,-p)\) for all \( p\in \N\).

The link invariant in \(\Uzerokk{k+k}\) will respect the Alexander skein relation, with \(t=q^{-2}\):
\begin{equation}
\begin{tikzpicture}[scale=.75,anchorbase]
\draw [thick,->] (0,0) -- (1,1);
\draw [thick] (1,0) -- (.6,.4);
\draw [thick,->] (.4,.6) -- (0,1);
\end{tikzpicture}
-
\begin{tikzpicture}[scale=.75,anchorbase]
\draw [thick,->] (1,0) -- (0,1);
\draw [thick] (0,0) -- (.4,.4);
\draw [thick,->] (.6,.6) -- (1,1);
\end{tikzpicture}
-(q^{-1}-q)\;
\begin{tikzpicture}[scale=.75,anchorbase]
\draw [thick,->] (1,0) -- (1,1);
\draw [thick,->] (0,0) -- (0,1);
\end{tikzpicture}\;
\mapsto 0.
\label{Alex_Skein}
\end{equation}
As in \cite{SarAlexander}, one runs into the issue of having the value zero associated to the unknot (since now \([\beta]=[0]=0\)), and, in turn, to any knot. The trick \cite{GPT,SarAlexander} is to cut open one of the strands, but the arguments allowing one to do this in the representation theory of \(U_q(\gl(1|1))\) need to be lifted up abstractly in \(\Uzerokk{k+l}\).

\begin{lemma} \label{lemma_OneDim}
The endomorphism space \(\End_{\Uzerokk{k+l}}(\onell{(1,0,\dots,0)})\) with \(k\geq 1\) is 
 \(1\)-dimensional.
\end{lemma}

\begin{proof}
Let us consider a monomial in \(\End_{\Uzerokk{k+l}}(\onell{(1,0,\dots,0)})\), corresponding diagrammatically to a ladder web \(\boldsymbol{w}\). We possibly increase the value of \(k\), in order to be able to perform all necessary moves.

Let us thus consider the web \(\boldsymbol{w}\). Using isotopies in one or the other side and the proof of the first Markov's move (in particular Move \(\boldsymbol{(5)}\)),
we can assume that all trivalent vertices are on the left side. A typical situation now looks like:
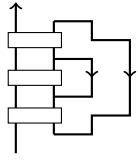
\[
\begin{tikzpicture}[scale=.5,anchorbase]
\draw [thick] (0,0) -- (0,.8);
\draw [thick] (1,0.5) -- (1,.8);
\draw (-.2,.8) rectangle (1.2,1.2);
\draw [thick] (0,1.2) -- (0,1.8);
\draw [thick] (1,1.2) -- (1,1.8);
\draw (-.2,1.8) rectangle (1.2,2.2);
\draw [thick] (0,2.2) -- (0,2.8);
\draw [thick] (1,2.2) -- (1,2.8);
\draw (-.2,2.8) rectangle (1.2,3.2);
\draw [thick, ->] (0,3.2) -- (0,4);
\draw [thick] (1,3.2) -- (1,3.5);
\draw [thick, directed=.5] (1,3.5) -- (2,3.5) -- (2,3) -- (3,3) -- (3,1) -- (2,1) -- (2,.5) -- (1,.5);
\draw [thick, directed=.5] (1,2.5) -- (2,2.5) -- (2,1.5) -- (1,1.5);
\end{tikzpicture}
\]
Let us focus on an inner-most right curl. We can use the same moves as in the proof of Lemma~\ref{lemma_circles_HomflyPt} to move up all \(E_{k-1}\)'s that are comprised between the \(F_k^{(a)}\) and the \(E_{k}^{(b)}\) that form the curl (so that they commute with \(F_k^{(a)}\)) and the \(F_{k-1}\) to the bottom (so that they commute with \(E_k^{(b)}\)). At the end, we are left with
\[
\begin{tikzpicture}[scale=.5,baseline={([yshift=-0.5ex]0,0.5)}]
\draw [thick, ->] (0,0) -- (0,2);
\draw [thick, directed=.5] (0,1.5) -- (1,1.5) -- (1,.5) -- (0,.5);
\draw [semithick,dotted] (1,2) -- (1,1.5);
\draw [semithick,dotted, ->] (1,.5) -- (1,0);
\node at (0,-.3) {\tiny \(k\)};
\node at (.5,1.8) {\tiny \(l\)};
\end{tikzpicture}
\; =
\qbin{-k}{l}
\;\begin{tikzpicture}[scale=.5,baseline={([yshift=-0.5ex]0,0.5)}]
\draw [thick, ->] (0,0) -- (0,2);
\draw [semithick,dotted, ->] (1,2) -- (1,0);
\node at (.5,1.8) {};
\node at (0,-.3) {\tiny \(k\)};
\end{tikzpicture}
\;=
(-1)^l\qbin{k+l-1}{l}\;
\begin{tikzpicture}[scale=.5,baseline={([yshift=-0.5ex]0,0.5)}]
\draw [thick, ->] (0,0) -- (0,2);
\draw [semithick,dotted, ->] (1,2) -- (1,0);
\node at (.5,1.8) {};
\node at (0,-.3) {\tiny \(k\)};
\end{tikzpicture}
\]
The last equality (and all details about the quantum binomial coefficients) can be found in \cite[Section~1.3]{Lus4}.
Inductively we can thus keep reducing the curls until there are no left, and we only have an upward web on the left side of the picture. Since this upward web is an endomorphism of \((1,0,\dotsc,0)\), this can only be a line, which can be isotoped to an identity.
\end{proof}

\begin{remark}
The same proof  shows that \(\End_{\UzklB}\onell{(\lambda_1,\dots,\lambda_k,\beta-0,\dots,\beta-0)}\) cannot be of higher dimension than \(\End_{\Uzmm{k}}\onell{(\lambda_1,\dots,\lambda_k)}\).
\end{remark}

\begin{corollary}
\(\End_{\UzklB}(\onell{(1,1,0,\dots,0,\beta-0,\dots,\beta-0)})\) is \(2\)-dimensional, generated by \(\onell{(1,1,0,\dots)}\) and \(E_1F_1\onell{(1,1,0,\dots)}\).
\end{corollary}

The usual decomposition on the corresponding tensor product of representations can  be lifted to \(\UzklB\): setting \(e_1=\frac{1}{[2]}E_1F_1\onell{(1,1,0,\dots)}\) and \(e_2=\onell{(1,1,0,\dots)}-e_1\) we obtain  an orthogonal idempotent decomposition of the identity. Notice that that \(e_1\) and \(e_2\) are central idempotents, since \(\End_{\UzklB} ( \onell{(1,1,0,\dots)})\) is generated by \(e_1\) and \(e_2\).

\begin{prop} \label{cor_InvOpenStrand}
Let \(\boldsymbol{\tau}\in \End_{\UzklB}( \onell{(0,\dots,0,1,1,\beta-0,\dots)})\). Then  
\begin{equation}\label{eq:17}
\begin{tikzpicture}[scale=.5,anchorbase]
\draw [thick] (0,0) -- (0,1.5);
\draw  (-.2,1.5)  rectangle (1.2,2.5);
\node at (.5,2) {\(\boldsymbol{\tau}\)};
\draw [thick, ->] (0,2.5) -- (0,4);
\draw [thick, directed=.5] (1,2.5) -- (1,3.5) -- (2,3.5) -- (2,.5) -- (1,.5) -- (1,1.5);
\end{tikzpicture}\;
=\;
\begin{tikzpicture}[scale=.5,anchorbase]
\draw [thick] (0,0) -- (0,.5) -- (.4,.9);
\draw [thick] (.6,1.1) -- (1,1.5);
\draw  (-.2,1.5)  rectangle (1.2,2.5);
\node at (.5,2) {\(\boldsymbol{\tau}\)};
\draw [thick] (1,2.5) -- (.6,2.9);
\draw [thick, ->] (.4,3.1) -- (0,3.5) -- (0,4);
\draw [thick, directed=.5] (0,2.5) --  (1,3.5) -- (2,3.5) -- (2,.5) -- (1,.5) -- (0,1.5);
\end{tikzpicture}
\end{equation}
\end{prop}
\begin{proof}
This proof is similar to the proof of \cite[Proposition~4.5]{SarAlexander}.
Diagrammatically, we have
\[
\begin{tikzpicture}[scale=.5,anchorbase]
\draw [thick] (0,0) -- (0,1.5);
\draw  (-.2,1.5)   rectangle (1.2,2.5);
\node at (.5,2) {\(\boldsymbol{b}\)};
\draw [thick, ->] (0,2.5) -- (0,4);
\draw [thick, directed=.5] (1,2.5) -- (1,3.5) -- (2,3.5) -- (2,.5) -- (1,.5) -- (1,1.5);
\end{tikzpicture}
\;=\;
\begin{tikzpicture}[scale=.5,anchorbase]
\draw [thick] (0,0) -- (0,.7);
\draw  (-.2,.7)   rectangle (1.2,1.3);
\node at (.5,1) {\(e_1\)};
\draw [thick] (0,1.3) -- (0,1.5);
\draw [thick] (1,1.3) -- (1,1.5);
\draw  (-.2,1.5)   rectangle (1.2,2.5);
\node at (.5,2) {\(\boldsymbol{b}\)};
\draw [thick] (0,2.5) -- (0,2.7);
\draw [thick] (1,2.5) -- (1,2.7);
\draw (-.2,2.7)  rectangle (1.2,3.3);
\node at (.5,3) {\(e_1\)};
\draw [thick, ->] (0,3.3) -- (0,4);
\draw [thick, directed=.5] (1,3.3) -- (1,3.5) -- (2,3.5) -- (2,.5) -- (1,.5) -- (1,.7);
\end{tikzpicture}
\;+\;
\begin{tikzpicture}[scale=.5,anchorbase]
\draw [thick] (0,0) -- (0,.7);
\draw  (-.2,.7)   rectangle (1.2,1.3);
\node at (.5,1) {\(e_2\)};
\draw [thick] (0,1.3) -- (0,1.5);
\draw [thick] (1,1.3) -- (1,1.5);
\draw  (-.2,1.5)   rectangle (1.2,2.5);
\node at (.5,2) {\(\boldsymbol{b}\)};
\draw [thick] (0,2.5) -- (0,2.7);
\draw [thick] (1,2.5) -- (1,2.7);
\draw (-.2,2.7)  rectangle (1.2,3.3);
\node at (.5,3) {\(e_1\)};
\draw [thick, ->] (0,3.3) -- (0,4);
\draw [thick, directed=.5] (1,3.3) -- (1,3.5) -- (2,3.5) -- (2,.5) -- (1,.5) -- (1,.7);
\end{tikzpicture}
\;+\;
\begin{tikzpicture}[scale=.5,anchorbase]
\draw [thick] (0,0) -- (0,.7);
\draw  (-.2,.7)   rectangle (1.2,1.3);
\node at (.5,1) {\(e_1\)};
\draw [thick] (0,1.3) -- (0,1.5);
\draw [thick] (1,1.3) -- (1,1.5);
\draw  (-.2,1.5)   rectangle (1.2,2.5);
\node at (.5,2) {\(\boldsymbol{b}\)};
\draw [thick] (0,2.5) -- (0,2.7);
\draw [thick] (1,2.5) -- (1,2.7);
\draw (-.2,2.7)  rectangle (1.2,3.3);
\node at (.5,3) {\(e_2\)};
\draw [thick, ->] (0,3.3) -- (0,4);
\draw [thick, directed=.5] (1,3.3) -- (1,3.5) -- (2,3.5) -- (2,.5) -- (1,.5) -- (1,.7);
\end{tikzpicture}
\;+\;
\begin{tikzpicture}[scale=.5,anchorbase]
\draw [thick] (0,0) -- (0,.7);
\draw  (-.2,.7)   rectangle (1.2,1.3);
\node at (.5,1) {\(e_2\)};
\draw [thick] (0,1.3) -- (0,1.5);
\draw [thick] (1,1.3) -- (1,1.5);
\draw  (-.2,1.5)   rectangle (1.2,2.5);
\node at (.5,2) {\(\boldsymbol{b}\)};
\draw [thick] (0,2.5) -- (0,2.7);
\draw [thick] (1,2.5) -- (1,2.7);
\draw (-.2,2.7)  rectangle (1.2,3.3);
\node at (.5,3) {\(e_2\)};
\draw [thick, ->] (0,3.3) -- (0,4);
\draw [thick, directed=.5] (1,3.3) -- (1,3.5) -- (2,3.5) -- (2,.5) -- (1,.5) -- (1,.7);
\end{tikzpicture}
\]
Since \(e_1,e_2\) are central orthogonal idempotents, the two middle terms are zero.

Let us now show that we get the same result for the r.h.s.\ of \eqref{eq:17}. We have 
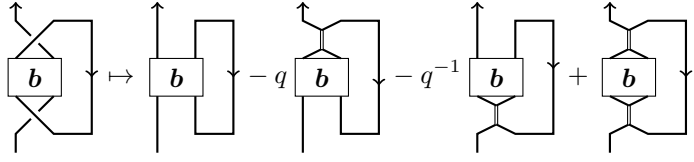
\[
\begin{tikzpicture}[scale=.5,anchorbase]
\draw [thick] (0,0) -- (0,.5) -- (.4,.9);
\draw [thick] (.6,1.1) -- (1,1.5);
\draw  (-.2,1.5)   rectangle (1.2,2.5);
\node at (.5,2) {\(\boldsymbol{b}\)};
\draw [thick] (1,2.5) -- (.6,2.9);
\draw [thick, ->] (.4,3.1) -- (0,3.5) -- (0,4);
\draw [thick, directed=.5] (0,2.5) --  (1,3.5) -- (2,3.5) -- (2,.5) -- (1,.5) -- (0,1.5);
\end{tikzpicture}
\;\mapsto\;
\begin{tikzpicture}[scale=.5,anchorbase]
\draw [thick] (0,0) -- (0,1.5);
\draw  (-.2,1.5)   rectangle (1.2,2.5);
\node at (.5,2) {\(\boldsymbol{b}\)};
\draw [thick, ->] (0,2.5)  -- (0,4);
\draw [thick, directed=.5] (1,2.5) --  (1,3.5) -- (2,3.5) -- (2,.5) -- (1,.5) -- (1,1.5);
\end{tikzpicture}
\;-q\;
\begin{tikzpicture}[scale=.5,anchorbase]
\draw [thick] (0,0) -- (0,1.5);
\draw  (-.2,1.5)   rectangle (1.2,2.5);
\node at (.5,2) {\(\boldsymbol{b}\)};
\draw [thick] (1,2.5) -- (.5,2.75);
\draw [thick] (0,2.5) -- (.5,2.75);
\draw [double] (.5,2.75) -- (.5,3.25);
\draw [thick, ->] (.5,3.25) -- (0,3.5) -- (0,4);
\draw [thick, directed=.5] (.5,3.25) -- (1,3.5) -- (2,3.5) -- (2,.5) -- (1,.5) -- (1,1.5);
\end{tikzpicture}
\;-q^{-1}\;
\begin{tikzpicture}[scale=.5,anchorbase]
\draw [thick] (0,0) -- (0,.5) -- (.5,.75);
\draw [double] (.5,.75) -- (.5,1.25);
\draw [thick] (.5,1.25) -- (0,1.5);
\draw [thick] (.5,1.25) -- (1,1.5);
\draw  (-.2,1.5)   rectangle (1.2,2.5);
\node at (.5,2) {\(\boldsymbol{b}\)};
\draw [thick, ->] (0,2.5) -- (0,4);
\draw [thick, directed=.5] (1,2.5) -- (1,3.5) -- (2,3.5) -- (2,.5) -- (1,.5) -- (.5,.75);
\end{tikzpicture}
\;+\;
\begin{tikzpicture}[scale=.5,anchorbase]
\draw [thick] (0,0) -- (0,.5) -- (.5,.75);
\draw [double] (.5,.75) -- (.5,1.25);
\draw [thick] (.5,1.25) -- (0,1.5);
\draw [thick] (.5,1.25) -- (1,1.5);
\draw  (-.2,1.5)   rectangle (1.2,2.5);
\node at (.5,2) {\(\boldsymbol{b}\)};
\draw [thick] (1,2.5) -- (.5,2.75);
\draw [thick] (0,2.5) -- (.5,2.75);
\draw [double] (.5,2.75) -- (.5,3.25);
\draw [thick, ->] (.5,3.25) -- (0,3.5) -- (0,4);
\draw [thick, directed=.5] (.5,3.25) -- (1,3.5) -- (2,3.5) -- (2,.5) -- (1,.5) -- (.5,.75);
\end{tikzpicture}
\]
Since \(e_2E_1F_1\onell{(1,1,0,\dots)}=0=E_1F_1e_2\onell{(1,1,0,\dots)}\),  the previous expression equals
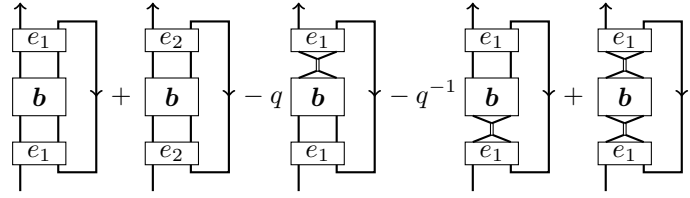
\[
\begin{tikzpicture}[scale=.5,anchorbase]
\draw [thick] (0,0) -- (0,.7);
\draw  (-.2,.7)   rectangle (1.2,1.3);
\node at (.5,1) {\(e_1\)};
\draw [thick] (0,1.3) -- (0,2);
\draw [thick] (1,1.3) -- (1,2);
\draw  (-.2,2)   rectangle (1.2,3);
\node at (.5,2.5) {\(\boldsymbol{b}\)};
\draw [thick] (0,3) -- (0,3.7);
\draw [thick] (1,3) -- (1,3.7);
\draw (-.2,3.7)  rectangle (1.2,4.3);
\node at (.5,4) {\(e_1\)};
\draw [thick, ->] (0,4.3) -- (0,5);
\draw [thick, directed=.5] (1,4.3) -- (1,4.5) -- (2,4.5) -- (2,.5) -- (1,.5) -- (1,.7);
\end{tikzpicture}
\;+\;
\begin{tikzpicture}[scale=.5,anchorbase]
\draw [thick] (0,0) -- (0,.7);
\draw  (-.2,.7)   rectangle (1.2,1.3);
\node at (.5,1) {\(e_2\)};
\draw [thick] (0,1.3) -- (0,2);
\draw [thick] (1,1.3) -- (1,2);
\draw  (-.2,2)   rectangle (1.2,3);
\node at (.5,2.5) {\(\boldsymbol{b}\)};
\draw [thick] (0,3) -- (0,3.7);
\draw [thick] (1,3) -- (1,3.7);
\draw (-.2,3.7)  rectangle (1.2,4.3);
\node at (.5,4) {\(e_2\)};
\draw [thick, ->] (0,4.3) -- (0,5);
\draw [thick, directed=.5] (1,4.3) -- (1,4.5) -- (2,4.5) -- (2,.5) -- (1,.5) -- (1,.7);
\end{tikzpicture}
\;-q\;
\begin{tikzpicture}[scale=.5,anchorbase]
\draw [thick] (0,0) -- (0,.7);
\draw  (-.2,.7)   rectangle (1.2,1.3);
\node at (.5,1) {\(e_1\)};
\draw [thick] (0,1.3) -- (0,2);
\draw [thick] (1,1.3) -- (1,2);
\draw  (-.2,2)   rectangle (1.2,3);
\node at (.5,2.5) {\(\boldsymbol{b}\)};
\draw [thick] (0,3) -- (.5,3.2);
\draw [thick] (1,3) -- (.5,3.2);
\draw [double] (.5,3.2) -- (.5,3.5);
\draw [thick] (.5,3.5) -- (0,3.7);
\draw [thick] (.5,3.5) -- (1,3.7);
\draw (-.2,3.7)  rectangle (1.2,4.3);
\node at (.5,4) {\(e_1\)};
\draw [thick, ->] (0,4.3) -- (0,5);
\draw [thick, directed=.5] (1,4.3) -- (1,4.5) -- (2,4.5) -- (2,.5) -- (1,.5) -- (1,.7);
\end{tikzpicture}
\;-q^{-1}\;
\begin{tikzpicture}[scale=.5,anchorbase]
\draw [thick] (0,0) -- (0,.7);
\draw  (-.2,.7)   rectangle (1.2,1.3);
\node at (.5,1) {\(e_1\)};
\draw [thick] (0,1.3) -- (.5,1.5);
\draw [thick] (1,1.3) -- (.5,1.5);
\draw [double] (.5,1.5) -- (.5,1.8);
\draw [thick] (.5,1.8) -- (0,2);
\draw [thick] (.5,1.8) -- (1,2);
\draw  (-.2,2)   rectangle (1.2,3);
\node at (.5,2.5) {\(\boldsymbol{b}\)};
\draw [thick] (0,3) -- (0,3.7);
\draw [thick] (1,3) -- (1,3.7);
\draw (-.2,3.7)  rectangle (1.2,4.3);
\node at (.5,4) {\(e_1\)};
\draw [thick, ->] (0,4.3) -- (0,5);
\draw [thick, directed=.5] (1,4.3) -- (1,4.5) -- (2,4.5) -- (2,.5) -- (1,.5) -- (1,.7);
\end{tikzpicture}
\;+\;
\begin{tikzpicture}[scale=.5,anchorbase]
\draw [thick] (0,0) -- (0,.7);
\draw  (-.2,.7)   rectangle (1.2,1.3);
\node at (.5,1) {\(e_1\)};
\draw [thick] (0,1.3) -- (.5,1.5);
\draw [thick] (1,1.3) -- (.5,1.5);
\draw [double] (.5,1.5) -- (.5,1.8);
\draw [thick] (.5,1.8) -- (0,2);
\draw [thick] (.5,1.8) -- (1,2);
\draw  (-.2,2)   rectangle (1.2,3);
\node at (.5,2.5) {\(\boldsymbol{b}\)};
\draw [thick] (0,3) -- (.5,3.2);
\draw [thick] (1,3) -- (.5,3.2);
\draw [double] (.5,3.2) -- (.5,3.5);
\draw [thick] (.5,3.5) -- (0,3.7);
\draw [thick] (.5,3.5) -- (1,3.7);
\draw (-.2,3.7)  rectangle (1.2,4.3);
\node at (.5,4) {\(e_1\)};
\draw [thick, ->] (0,4.3) -- (0,5);
\draw [thick, directed=.5] (1,4.3) -- (1,4.5) -- (2,4.5) -- (2,.5) -- (1,.5) -- (1,.7);
\end{tikzpicture}
\]
Using \(e_1E_1F_1\onell{(1,1,0,\dots)}=[2]e_1\onell{(1,1,0,\dots)}=E_1F_1e_1\onell{(1,1,0,\dots)}\), this becomes
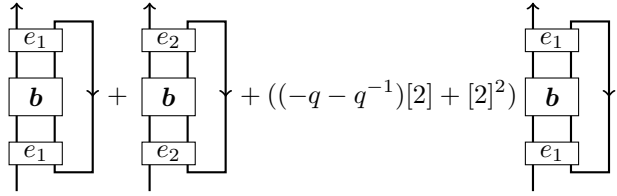
\[
\begin{tikzpicture}[scale=.5,anchorbase]
\draw [thick] (0,0) -- (0,.7);
\draw  (-.2,.7)   rectangle (1.2,1.3);
\node at (.5,1) {\(e_1\)};
\draw [thick] (0,1.3) -- (0,2);
\draw [thick] (1,1.3) -- (1,2);
\draw  (-.2,2)   rectangle (1.2,3);
\node at (.5,2.5) {\(\boldsymbol{b}\)};
\draw [thick] (0,3) -- (0,3.7);
\draw [thick] (1,3) -- (1,3.7);
\draw (-.2,3.7)  rectangle (1.2,4.3);
\node at (.5,4) {\(e_1\)};
\draw [thick, ->] (0,4.3) -- (0,5);
\draw [thick, directed=.5] (1,4.3) -- (1,4.5) -- (2,4.5) -- (2,.5) -- (1,.5) -- (1,.7);
\end{tikzpicture}
\;+\;
\begin{tikzpicture}[scale=.5,anchorbase]
\draw [thick] (0,0) -- (0,.7);
\draw  (-.2,.7)   rectangle (1.2,1.3);
\node at (.5,1) {\(e_2\)};
\draw [thick] (0,1.3) -- (0,2);
\draw [thick] (1,1.3) -- (1,2);
\draw  (-.2,2)   rectangle (1.2,3);
\node at (.5,2.5) {\(\boldsymbol{b}\)};
\draw [thick] (0,3) -- (0,3.7);
\draw [thick] (1,3) -- (1,3.7);
\draw (-.2,3.7)  rectangle (1.2,4.3);
\node at (.5,4) {\(e_2\)};
\draw [thick, ->] (0,4.3) -- (0,5);
\draw [thick, directed=.5] (1,4.3) -- (1,4.5) -- (2,4.5) -- (2,.5) -- (1,.5) -- (1,.7);
\end{tikzpicture}
\;+((-q-q^{-1})[2]+[2]^2)\;
\begin{tikzpicture}[scale=.5,anchorbase]
\draw [thick] (0,0) -- (0,.7);
\draw  (-.2,.7)   rectangle (1.2,1.3);
\node at (.5,1) {\(e_1\)};
\draw [thick] (0,1.3) -- (0,2);
\draw [thick] (1,1.3) -- (1,2);
\draw  (-.2,2)   rectangle (1.2,3);
\node at (.5,2.5) {\(\boldsymbol{b}\)};
\draw [thick] (0,3) -- (0,3.7);
\draw [thick] (1,3) -- (1,3.7);
\draw (-.2,3.7)  rectangle (1.2,4.3);
\node at (.5,4) {\(e_1\)};
\draw [thick, ->] (0,4.3) -- (0,5);
\draw [thick, directed=.5] (1,4.3) -- (1,4.5) -- (2,4.5) -- (2,.5) -- (1,.5) -- (1,.7);
\end{tikzpicture}
\]
and we are done.
\end{proof}

We finally obtain the following:

\begin{theorem} \label{prop:Alexander}
Let \(\boldsymbol{b}\) be a braid whose closure is a link \(K\). Let \(\tilde{\boldsymbol{b}}\) denote the closure of all but one strands of \(\boldsymbol{b}\). Then the image of \(\tilde{\boldsymbol{b}}\) in 
 \(\End_{\Uzerokk{k+k-1}}(\onell{(0,\dots, 0,1,0,\dots, \beta-0,\dots)})\), where the \(1\) appears in the same position as the open strand, is the Alexander polynomial of \(K\).
\end{theorem}
\begin{proof} 
It follows from Lemma~\ref{lemma_OneDim} that the endomorphism space is one-dimensional, hence we  get a polynomial.
By Proposition~\ref{cor_InvOpenStrand}, this does not depend on  the chosen strand.
Exactly as in \cite[Theorem 4.6]{SarAlexander} one can prove that it also does not depend on  the cutting place in the strand, and the fact we have a knot invariant. 
By \eqref{Alex_Skein}, this knot invariant satisfies the skein formula of the Alexander polynomial, and by  Lemma~\ref{lemma_OneDim} its value on the unknot coincides with the Alexander polynomial. Hence this invariant is the Alexander polynomial.
\end{proof}

Note that the process we developed to recover the Alexander polynomial is not \emph{stricto sensu} the specialization of the process yielding the Homfly-Pt polynomial: there is the extra step of cutting open one strand. However, the proofs of this last section are not specific to the \(\beta=0\) specialization, and it is easy to prove the following proposition by following the same steps as above. One thus can obtain all of these invariants from the exact same picture.

\begin{prop} \label{prop:reducedSpecialization}
Let \(K\) be a link, and let \(\tilde{\boldsymbol{b}}\) be as in Theorem~\ref{prop:Alexander}. Then
\begin{enumerate}[(i)]
\item the image of \(\tilde{\boldsymbol{b}}\)
  in
  \(\End_{\UzmmB{(k)+(k-1)}}(\onell{(0,\dots, 0,1,0,\dots,
    \beta-0,\dots,\beta-0)})\),
   gives the
  reduced Homfly-Pt polynomial of \(K\) (normalized to \(1\)
  for the unknot), and
\item  the image of \(\tilde{\boldsymbol{b}}\) 
 in
  \(\End_{\U_q(\gl_{(k)+(k-1)})_{\beta=m}}(\onell{(0,\dots,
    0,1,0,\dots, \beta-0,\dots,\beta-0)})\)
  gives the reduced \(\slm\) invariant of \(K\).
\end{enumerate}
\end{prop}



\newcommand{\etalchar}[1]{$^{#1}$}
\providecommand{\bysame}{\leavevmode\hbox to3em{\hrulefill}\thinspace}
\providecommand{\MR}{\relax\ifhmode\unskip\space\fi MR }
\providecommand{\MRhref}[2]{%
  \href{http://www.ams.org/mathscinet-getitem?mr=#1}{#2}
}
\providecommand{\href}[2]{#2}

\end{document}